\documentclass [12pt,a4paper]{amsart}
\usepackage[utf8]{inputenc}
\usepackage{amsfonts,amsmath,amsthm}
\usepackage{amssymb,latexsym}

\usepackage{latexsym}
\usepackage{mathrsfs}
\usepackage[hypcap]{caption}

\usepackage{enumitem}

\newtheorem{theorem}{Theorem}[section]

\newtheorem{lemma}{Lemma}[section]

\newtheorem{corollary}{Corollary}[section]
\newtheorem{proposition}[theorem]{Proposition}
\usepackage{hyperref}

\usepackage[color=red!40]{todonotes}

\usepackage{tikz}
\usepackage{tikz-cd}
\usetikzlibrary{arrows, matrix}

 \newcommand\vol{\mathrm{vol}}
  \newcommand\area{\mathrm{area}}

\newcommand\sys{\mathrm{sys}_{1}}
\newcommand\sysg{\mathrm{sysg}}

\DeclareMathOperator{\SL}{SL}
\DeclareMathOperator{\N}{N}
\DeclareMathOperator{\GL}{GL}
\DeclareMathOperator{\SO}{SO}

\DeclareMathOperator{\inj}{inj}

\DeclareMathOperator{\Spin}{Spin}
\DeclareMathOperator{\arccosh}{arccosh}
\DeclareMathOperator{\Isom}{Isom}

\begin{document}

\title[systole of congruence coverings]{{\normalsize Systole of congruence coverings of arithmetic hyperbolic manifolds}}

\author[Plinio G. P. Murillo]{Plinio G. P. Murillo}

\thanks{\textit{Date:} \today}

\thanks{This work was supported by CAPES and CNPq}

\maketitle
\vspace{-3mm}
\begin{center}
With an appendix by Cayo Dória and Plinio G. P. Murillo
\end{center}
\vspace{5mm}

\begin{abstract}
In this paper we prove that, for any arithmetic hyperbolic $n$-manifold $M$ of the first type, the systole of most of the principal congruence coverings $M_{I}$ satisfy $$\sys(M_{I})\geq \frac{8}{n(n+1)}\log(\vol(M_{I}))-c,$$ where $c$ is a constant independent of $I$. This generalizes previous work of Buser and Sarnak, and Katz, Schaps and Vishne in dimension 2 and 3. As applications, we obtain explicit estimates for systolic genus of hyperbolic manifolds studied by Belolipetsky and the distance of homological codes constructed by Guth and Lubotzky. In an appendix together with Cayo Dória we prove that the constant $\frac{8}{n(n+1)}$ is sharp.

\end{abstract}  

\section{Introduction}

The {\it systole} of a Riemannian manifold $M$ is the length of a shortest non-contractible closed geodesic in $M$ and it is denoted by $\sys(M)$. In 1994, P. Buser and P. Sarnak constructed in \cite{BS94} the first explicit examples of surfaces with systole growing logarithmically with the genus. For this construction they used sequences of congruence coverings of an arithmetic compact Riemann surface. These examples were generalized in 2007 by M. Katz, M. Schaps and U. Vishne to congruence coverings of any compact arithmetic Riemann surfaces and arithmetic hyperbolic 3-manifolds \cite{KSV07}. They proved that any sequence of congruence covering $S_{i}$ of a compact arithmetic Riemann surface $S$ satisfy  

\begin{equation}\label{previousresult}
\sys(S_{i})\geq\frac{4}{3}\log(\area(S_{i}))-c \hspace{4mm} \mbox{as}\hspace{2mm} i\rightarrow\infty,
\end{equation}

where $c$ is independent of $i$. In dimension 3, the corresponding result which was also obtained in \cite{KSV07} has the constant $\frac{2}{3}$ instead of $\frac{4}{3}$.\\

It is known that a sequence of regular congruence coverings of a compact arithmetic hyperbolic manifold attains asymptotically the logarithmic growth of the systole (see \cite[3.C.6]{Gro96} or \cite[Sec.4]{GL13}), but the examples above are the only cases in which the explicit constant in the systole growth was known so far. In particular, it would be interesting to understand how the asymptotic constant depends on the dimension.\\

The purpose of this paper is to generalize the previous results to the sequences of principal congruence coverings $M_{I}$ of an arithmetic hyperbolic manifold $M$ of the first type. We show that if the dimension of $M$ is $n$ under a suitable condition on the ideals $I$ those sequences
eventually satisfy

\begin{equation}\label{result}
\sys(M_{I})\geq \frac{8}{n(n+1)}\log(\vol(M_{I}))-c,
\end{equation} 
where the constant $c$ is independent of the ideal $I$. We refer to Sections \ref{Arithmetic subgroups of Spin(1,n)} and \ref{congruencesubgroups} for the definitions and Theorem \ref{maintheorem} for the precise statement of the result.\\

This result give us the first examples of explicit constants for the growth of systole of sequences of congruence coverings of arithmetic hyperbolic manifolds in dimensions greater than three. In \cite{Murillo16}, the author studies the same problem for other non-positively curved arithmetic manifolds.\\

We would like to remark the following aspect in our approach. In order to generalize to higher dimensions Inequality \eqref{previousresult}, the most natural approach could be to consider congruence subgroups $\overline{\Gamma}(I)$ of arithmetic subgroups $\overline{\Gamma}$ in $\SO(1,n)^{\circ}$, and to study the systole of the quotient spaces $\overline{\Gamma}(I)\backslash\mathbb{H}^{n}$. However, the congruence coverings $M_{I}$ of the hyperbolic manifolds $M$ appearing in Theorem \ref{maintheorem} arise from congruence subgroups of arithmetic subgroups $\Gamma$ in $\Spin(1,n)$. The main reason for this is that the constant in the lower bound of the systole obtained from Inequality \eqref{result} matches with the previously known constants $\frac{4}{3}$ and $\frac{2}{3}$ in dimensions 2 and 3 proved in \cite{KSV07}. On the other hand, if $\overline{\Gamma}$ is an arithmetic subgroup in $\SO(1,n)^{\circ}$ of the first type, we could only prove that if $I$ is a prime ideal with norm sufficiently large, then the congruence subgroups $\overline{\Gamma}(I)\subset \overline{\Gamma}$ satisfy a weaker lower bound

\begin{equation}
\sys(\overline{\Gamma}(I)\backslash\mathbb{H}^{n})\geq\frac{4}{n(n+1)}\log(\vol(\overline{\Gamma}(I)\backslash\mathbb{H}^{n}))-d,
\end{equation}

where $d$ is a constant independent of $I$. The proof of this result is based on slightly different ideas and the details can be found in the author's PhD thesis \cite{MurilloThesis}.\\

This paper is organized as follows. We begin in Section \ref{Preliminaries} recalling basic facts about hyperbolic manifolds and the group $\Spin(1,n)$. We also define in this section arithmetic hyperbolic manifolds $M$ of the first type and their congruence coverings $M_{I}$. We then study the displacement of the action of $\Spin(1,n)$ on $\mathbb{H}^{n}$ in Section \ref{displacement}, and we estimate the length of closed geodesics of $M_{I}$ in terms of the norm of the ideal $I$ in Section \ref{lowerbound}. In Section \ref{index} we relate $\vol(M_{I})$ with the norm of the ideal $I$, and in Section \ref{proofoftheorem} the main result is proved. In the last two sections we apply our result in two different contexts: in Section \ref{systolicgenus} we give a more precise relation between the systolic genus and the volume of congruence coverings of arithmetic hyperbolic manifolds studied by M. Belolipetsky in \cite{Bel13}. In Section \ref{homologicalcodes}, we present an explicit lower bound for the distance of homological codes constructed by L. Guth and A. Lubotzky using arithmetic hyperbolic manifolds \cite{GL13}.\\

\textbf{Acknowledgements}. I would like to thank Mikhail Belolipetsky for his advice, encouragement and valuable suggestions on the preliminary versions of this work. I want to thank Cayo Dória for many fruitful discussions and collaboration. I also want to thank Alan Reid, Larry Guth, Uzi Vishne and Vincent Emery for their interest in the work and valuable suggestions.

\section{Preliminary material}\label{Preliminaries}

\subsection{Hyperbolic manifolds}\label{Arithmetic subgroups of SO(1,n)}

The hyperbolic $n$-space is the complete simply connected $n$-dimensional Riemannian manifold with the constant sectional curvature equal to $-1$. The \textit{hyperboloid model} of the hyperbolic $n$-space is given by
$$\mathbb{H}^{n} =\lbrace x\in \mathbb{R}^{n+1}; x_{0}^{2}-  x_{1}^{2} -\cdots - x_{n-1}^{2}-x_ {n}^{2}=1,  x_{0}>0\rbrace$$

with the metric $ds^{2}=-dx_{0}^{2}+\cdots + dx_{n}^{2}+dx_{n}^{2}$.\\

A \textit{hyperbolic manifold} $M$ is a complete Riemannian manifold of constant sectional curvature equal to $-1$. These manifolds are the quotients spaces $M=\overline{\Gamma}\backslash\mathbb{H}^{n}$, where $\overline{\Gamma}$ is a discrete torsion-free subgroup of $\Isom^{+}(\mathbb{H}^{n})$.\\

The group $\SO(1,n)$ of linear transformations preserving a quadratic form of the signature $(1,n)$ over $\mathbb{R}$ acts by isometries on the hyperbolic \\ $n$-space $\mathbb{H}^{n}$. The identity component $G=~\SO(1,n)^{\circ}$ is a real Lie group isomorphic to the orientation-preserving isometries $\Isom^{+}(\mathbb{H}^{n})$. Both these groups act linearly on the hyperboloid model of $\mathbb{H}^{n}$. This action is transitive and we can identify $\mathbb{H}^{n}=G/K$, where $K$ is the stabilizer of a point by the action of $G$.\\

A discrete subgroup $\Gamma$ of a Lie group $G$ is called a \textit{lattice} if the quotient $\Gamma \backslash G$ has finite measure with respect to a Haar measure of $G$. By the compactness of the subgroup $K\subset\SO(1,n)^{\circ}$, $\Gamma$ is a lattice in $\SO(1,n)^{\circ}$ if and only if the hyperbolic manifold $\Gamma\backslash\mathbb{H}^{n}$ has finite volume.

\subsection{The group $\textbf{Spin}$}\label{Clifford algebras and spin group}

The universal covering of $\SO(1,n)^{\circ}$ is the spin group $\Spin(1,n)$. In this section we recall the basic facts about the group $\Spin_{f}$ associated to a quadratic form $f$. We will refer the reader to \cite{Bourbaki59}, \cite{Chevalley54} and \cite{Dieudonne71} for further details.\\

Let $k$ be a field of characteristic not equal to $2$, and let $E$ be an $n$-dimensional vector space over $k$. Suppose that $f:E\rightarrow k$ is a non-degenerate quadratic form with associated bilinear form $\Phi$. If $T(E)$ denotes the tensor algebra of $E$ and $\mathfrak{a}_{f}$ the two-sided ideal of $T(E)$ generated by the elements $x\otimes y + y\otimes x-2\Phi(x,y)$, the $\textit{Clifford algebra of f}$ is defined as the quotient $\mathscr{C}(f,k)=T(E)/\mathfrak{a}_{f}$.\\

The Clifford algebra of $f$ is a unitary associative algebra over $k$ with a canonical map $j:E\rightarrow \mathscr{C}(f,k)$ such that $j(x)^{2}=f(x)$ for any $x\in E$, which satisfies the following universal property:  Given an associative $k$-algebra $R$ with unity, for any map $g:E\rightarrow R$ satisfying $g(x)^{2}=f(x)$ there exists a unique $k$-algebra homomorphism $\bar{g}:\mathscr{C}(f,k)\rightarrow R$ such that $\bar{g}\circ i=g$. \\

We identify $E$ with its image $j(E)$ in $\mathscr{C}(f,k)$ and $k$ with $k\cdot1\subset \mathscr{C}(f,k)$. If we choose an orthogonal basis $e_{1},\ldots,e_{n}$ of $E$ with respect to $\Phi$ then in $\mathscr{C}(f,k)$ we have the relations $e_{\nu}^{2}=f(e_{\nu})$ and $e_{\nu}e_{\mu}=-e_{\mu}e_{\nu}$ for $\mu,\nu =1,\ldots,n, \mu \neq \nu$. Let $\mathscr{A}_{n}$ be the set of subsets of the set $\lbrace 1,\ldots,n\rbrace$ and let $\mathscr{P}_{n}$ the subset of $\mathscr{A}_{n}$ given by the ordered sets $M=\lbrace \mu_{1},\ldots, \mu_{k}\rbrace\in\mathscr{A}_{n} $ with $\mu_{1}<\cdots <\mu_{k}$. For any $M\in\mathscr{P}_{n}$ we write $e_{M}=e_{\mu_{1}}\cdot\ldots\cdot e_{\mu_{k}}$ and $e_{\emptyset}=1$. Every element in $\mathscr{C}(f,k)$ can be written uniquely in the form $s=\Sigma_{M\in\mathscr{P}_{n}}s_{M}e_{M}$ with $s_{M}\in k.$ \\

The \textit{even Clifford algebra} $\mathscr{C}^{+}(f,k)$ of $f$ is the $k$-subalgebra of $\mathscr{C}(f,k)$ generated by the elements $e_{M}$ with $|M|$ even.  We call by the \textit{Clifford group} $\mathscr{C}l(f,k)$ of $f$ (resp. the \textit{special Clifford group} $\mathscr{C}l^{+}(f,k)$) the multiplicative group of the invertible elements of $\mathscr{C}(f,k)$ (resp. $\mathscr{C}^{+}(f,k)$) such that $sEs^{-1}=E$.\\

The Clifford algebra $\mathscr{C}(f,k)$ has an anti-automorphism $^{*}$ which acts on the basis elements by $\left( e_{\mu_{1}}e_{\mu_{2}}\cdots e_{\mu_{k}}\right)^{*}= e_{\mu_{k}}e_{\mu_{k-1}}\cdots e_{\mu_{1}}$. For any $s\in\mathscr{C}l^{+}(f,k)$, $ss^{*}$ is an element of $k^{\times}$, which is called the \textit{spinor norm} of $s$ \cite[§9, Proposition 4]{Bourbaki59}. The \textit{spin group of $f$} is defined as the group of elements in the special Clifford group with the spinor norm equal to one:
\begin{equation}
\Spin_{f}(k)=\lbrace s\in \mathscr{C}^{+}(f,k), sEs^{*}=E \hspace{2mm}\mbox{and}\hspace{2mm} ss^{*}=1\rbrace.
\end{equation}

%

The conditions defining the group $\Spin_{f}$ are determined by polynomials in $2^{n}$ variables which give $\Spin_{f}$ the structure of an affine $k$-algebraic group. If $k=\mathbb{R}$ and $f$ has signature $(1,n)$ over $\mathbb{R}$ the group $\Spin_{f}(\mathbb{R})$ is isomorphic to $\Spin(1,n)$.\\

For an element $s\in \Spin_{f}(k)$, the map $\varphi_{s}:E\mapsto E$ given by $\varphi_{s}(x)=sxs^{-1}$ defines a homomorphism

\begin{align}
\varphi: \Spin_{f}(k)&\mapsto \mbox{SO}_{f}(k), \label{phi} \\
s & \mapsto\varphi_{s}. \nonumber
\end{align}



The kernel of $\varphi$ is the set $\lbrace1,-1\rbrace$. Moreover if $f$ is isotropic then  

\begin{equation*}
\SO_{f}(k)/\varphi(\Spin_{f}(k))\simeq k^{\times}/(k^{\times})^2    
\end{equation*}

\cite[II.2.3, II.2.6, II.3.3 and II.3.7]{Chevalley54}. If $f$ is isotropic and $k$ is a finite field this implies that  $\lvert\Spin_{f}(k)\rvert=\lvert
\SO_{f}(k)\rvert.$ If $k=\mathbb{R}$ and $f$ is isotropic, we have moreover that the image of $\varphi$ is equal to the group $\SO_{f}(\mathbb{R})^{\circ}$ \cite[Sec. 2.9]{Chevalley54}. In particular, any lattice $\Gamma$ in $\Spin(1,n)$ acts on the hyperbolic $n$-space $\mathbb{H}^{n}$ and $\Gamma\backslash\mathbb{H}^{n}$ has finite volume.\\



We finish the section with a definition. Analogously to the complex numbers and the quaternion algebra, for $s=\Sigma_{M\in\mathscr{P}_{n}}s_{M}i_{M}$ we define the \textit{$k$-part of s} as $s_{k}:=s_{\emptyset}$. In the case $k=\mathbb{R}$ we call it the \textit{real part of}~$s$.

\subsection{Arithmetic hyperbolic manifolds of the first type}\label{Arithmetic subgroups of Spin(1,n)}
There is a wide class of discrete subgroups of a semi-simple Lie group $G$ which can be constructed using arithmetic tools. We recall that a discrete subgroup $\Gamma\subset G$ is \textit{arithmetic} if there exist a number field $k$, a
$k$-algebraic group $\mbox{H}$, and an epimorphism 
$\varphi: \mbox{H}(k \otimes_{\mathbb{Q}}\mathbb{R}) \rightarrow G$ with
compact kernel such that $\varphi(\mbox{H}(\mathcal{O}_{k}))$ is commensurable to $\Gamma$, where $\mathcal{O}_{k}$ is the ring of integers of $k$ and $\mbox{H}(\mathcal{O}_{k})$ denotes the $\mathcal{O}_{k}$-points of $\mbox{H}$ with respect to some fixed embedding of H into $\GL_{m}$. We call $k$ the \textit{field of definition} of $\Gamma$. \\

By a fundamental theorem of Borel and Harish-Chandra any arithmetic subgroup of a semi-simple Lie group is a lattice. A lattice of a semi-simple Lie group which is an arithmetic subgroup is called an \textit{arithmetic lattice}. A hyperbolic manifold $M=\Gamma\backslash\mathbb{H}^{n}$ such that $\Gamma$ is an arithmetic subgroup of $\Isom^{+}(\mathbb{H}^{n})\simeq\SO(1,n)^{\circ}$ is called an \textit{arithmetic hyperbolic manifold}. Arithmeticity is preserved by finite quotients, therefore if $\Gamma$ is an arithmetic lattice in $\Spin(1,n)$ then $\Gamma$ projects to an arithmetic lattice in $\SO(1,n)^{\circ}$.\\ 

The fact that the groups $\Spin_{f}$ and $\SO_{f}$ are algebraic groups allows us to construct a certain class of arithmetic subgroups of $\widetilde{G}=\Spin(1,n)$ and  $G=\SO(1,n)^{\circ}$. Suppose $k$ is a totally real number field and $f$ is a quadratic form defined over $k$ and of signature $(1,n)$ over $\mathbb{R}$, such that for any non-trivial embedding $\sigma: k\rightarrow\mathbb{R}$ the quadratic form $f^{\sigma}$ is positive definite. From now on, any quadratic form satisfying these conditions will be called an \textit{admissible quadratic form}. By restriction of scalars $\Spin_{f}(\mathcal{O}_{k})$ and  $\SO_{f}(\mathcal{O}_{k})$ embed as arithmetic subgroups of 
$\Spin_{f}(\mathbb{R})\simeq \widetilde{G}$ and $\SO_{f}(\mathbb{R})$, respectively. Intersecting with $\SO_{f}(\mathbb{R})^{\circ}$ we obtain an arithmetic subgroup of $\SO_{f}(\mathbb{R})^{\circ}\simeq G.$ The subgroups $\Gamma$ of $\widetilde{G}$ and $G$ constructed in this way and subgroups commensurable with them are called \textit{arithmetic lattices of the first type}. If $\Gamma$ is torsion-free,  $M=\Gamma\backslash\mathbb{H}^{n}$ is called an \textit{arithmetic hyperbolic manifold of the first type}.\\

Any non-cocompact arithmetic lattice in $G$ and $\widetilde{G}$ is of the first type and defined over $\mathbb{Q}$ \cite[Secs. 1-2]{LiMill93}, and any arithmetic subgroup defined over $\mathbb{Q}$ is non-cocompact if $n\geq 4$  \cite[Sec. §6.4]{Mor15}. If $n$ is even, all the co-compact arithmetic lattices of $\widetilde{G}$ and $G$ are of the first type. In odd dimensions $n\neq 7$ there is a second class of co-compact arithmetic subgroups of $G$ and $\widetilde{G}$ arising from skew-hermitian forms over division quaternion algebras. In dimension 7 there is a third class constructed using certain Cayley algebras (see \cite[Sec. 1]{LiMill93} and the references therein).\\

Defining arithmetic subgroups of $\Spin(1,n)$ requires to consider isomorphisms $\Spin_{f}(\mathbb{R})\simeq\Spin(1,n)$. We will need such an isomorphism preserving the real part of the spin elements. For completeness, we will finish this section describing one of these isomorphisms.\\  

Suppose $f$ is an admissible quadratic form of dimension $n+1$. Let $(E,f)$ be a quadratic vector space of dimension $n+1$ over $\mathbb{R}$ and $(\mathbb{R}^{n+1},q)$ be the Euclidean space with the quadratic form $q=x_{0}^{2}-x_{1}^{2}-\cdots-x_{n}^{2}$. Choose an orthogonal basis $i_{1},\ldots, i_{n+1}$ of $(E,f)$ such that $f(x)=x_{0}^{2}-x_{1}^{2}-\cdots-x_{n}^{2}$ and let $e_{1},\ldots, e_{n+1}$ be the canonical basis of $\mathbb{R}^{n+1}$. By the universal property of $\mathscr{C}(f,\mathbb{R})$, the map $g:E\rightarrow \mathscr{C}(q,\mathbb{R})$ defined as $g(\Sigma x_{j}i_{j})=\Sigma x_{j}e_{j}$ extends uniquely to an $\mathbb{R}$-algebras homomorphism $\tilde{g}:\mathscr{C}(f,\mathbb{R})\rightarrow\mathscr{C}(q,\mathbb{R})$ given by 

\begin{equation*}
\tilde{g}\left(\sum_{M\in\mathscr{P}_{n+1}}
s_{M}i_{M}\right)=
\sum_{M\in\mathscr{P}_{n+1}}s_{M}e_{M}.
\end{equation*}
From this equation it is clear that $\tilde{g}$ is an isomorphism and $\tilde{g}(s)_{\mathbb{R}}=s_{\mathbb{R}}$. Note that $\tilde{g}$ commutes with $*$ and restricts to an isomorphism of groups
\begin{equation*}
\tilde{g}:\Spin_{f}(\mathbb{R}) \xrightarrow{\sim}\Spin(1,n).
\end{equation*}

\subsection{Congruence coverings of arithmetic hyperbolic manifolds}\label{congruencesubgroups}

Let $\Gamma$ be an arithmetic subgroup of a semi-simple Lie group $G$ commensurable with $\varphi(\mbox{H}(\mathcal{O}_{k}))$ as in Section \ref{Arithmetic subgroups of Spin(1,n)}. If $I\subset\mathcal{O}_{k}$ is a non-zero ideal of $\mathcal{O}_{k}$, the \textit{principal congruence subgroup of} $\Gamma$ associated to $I$ is by definition the subgroup $\Gamma(I)=\Gamma\cap \varphi(\mbox{H}(I))$, where 

$$\mbox{H}(I):=\ker\big( \mbox{H}(\mathcal{O}_{k})\xrightarrow{\pi_{I}} \mbox{H}(\mathcal{O}_{k}/{I})\big)$$
and $\pi_{I}$ denotes the reduction map modulo $I$. We must to remark that the definition of the subgroup $\Gamma(I)$ depends on the representation of the group $\mbox{H}$ as a linear group, but its commensurability class does not depends on this choice.  \\ 

If $\Gamma$ is an arithmetic subgroup of $\Spin(1,n)$ and $M=\Gamma\backslash\mathbb{H}^{n}$, any ideal $I\subset\mathcal{O}_{k}$ defines a \textit{principal congruence covering} $M_{I}=\Gamma(I)\backslash\mathbb{H}^{n}\rightarrow M$. Since $\Gamma(I)$ is a normal finite-index subgroup of $\Gamma$, the covering $M_{I}\rightarrow M$ is a regular finite sheeted covering map.\\

We would like to give now a representation of the principal congruence subgroups of $\Gamma=\Spin_{f}(\mathcal{O}_{k})$ with $f$ an admissible quadratic form. We will use the multiplicative structure of the algebra $\mathscr{C}(f,\mathbb{R})$ to describe an embedding of $\Spin_{f}(\mathbb{R})$ into $\GL_{m}$ for some $m$. Choose an orthogonal basis $B=\lbrace e_{1},\ldots,e_{n+1}\rbrace$ with respect to $f$. The Clifford algebra $\mathscr{C}(f,\mathbb{R})$ is a real vector space of dimension $2^{n+1}$ with basis $\lbrace e_{M} \rbrace_{M\in\mathscr{P}_{n+1}}$ and the group $\Spin_{f}(\mathbb{R})$ acts on it by left multiplication. For any $s\in \Spin_{f}(\mathbb{R})$, the linear map $L_{s}(x)=sx$, $x\in \mathscr{C}(f,\mathbb{R})$, belongs to $\GL(\mathscr{C}(f,\mathbb{R}))\simeq \GL_{2^{n+1}}(\mathbb{R})$ and so we have a linear representation $L:\Spin_{f}(\mathbb{R})\rightarrow \GL_{2^{n+1}}(\mathbb{R})$. 
If $s=\sum_{\lvert M\rvert \hspace{1mm}\mbox{even}}s_{M} e_{M}$ with $s_{M}\in \mathbb{R}$, then $L_{s}\in \GL_{2^{n+1}}(\mathcal{O}_ {k})$ if and only if all $s_{M}\in\mathcal{O}_{k}$. We then obtain that

$$\Gamma=\Bigg\{ s=\sum_{\lvert M\rvert\hspace{1mm}\mbox{even}}s_{M} e_{M}\mid s_{M}\in \mathcal{O}_{k}  \hspace{2mm}\mbox{and}\hspace{2mm} ss^{*}=1 \Bigg\}.$$

With this representation, for an ideal $I\subset\mathcal{O}_{k}$ the principal congruence subgroup $\Gamma(I)$ corresponds to the kernel of the projection map $\Spin_{f}(\mathcal{O}_{k})\xrightarrow{\pi_{I}} \GL_{2^{n+1}}(\mathcal{O}_{k}/I)$, which corresponds to the group

$$\Gamma(I)= \Bigg\{s=\sum_{\lvert M\rvert\hspace{1mm}\mbox{even}}s_{M} e_{M}\in\Gamma \mid s_{M}\in I \hspace{2mm}\mbox{for}\hspace{2mm} M\neq \emptyset\hspace{2mm} \mbox{and} \hspace{1mm} s_{\mathbb{R}}-1\in I\Bigg\}.$$

\section{The displacement of elements in $\Spin(1,n)$ acting on $\mathbb{H}^{n}$}\label{displacement}
The main goal in this article is to obtain a lower bound for the systole of the hyperbolic manifolds $M_{I}=\Gamma(I)\backslash \mathbb{H}^{n}$. The geometry of $M_{I}$ is defined by the geometry of $\mathbb{H}^{n}$, the group $\Gamma(I)$ and its action on  $\mathbb{H}^{n}$. The lengths of closed geodesics on $M_{I}$ are then encapsulated by the hyperbolic distance between points $p\in\mathbb{H}^{n}$ and its displacement, or image, by the action of elements in $\Gamma(I)$.\\

In this section we start to explore this action seeing that the real part of elements in $\Spin(1,n)$ plays a remarkable role. In the next section we will specialize to the congruence subgroups $\Gamma(I)$.\\

Since we will focus our attention on the group $\Spin(1,n)$, we can fix the real vector space $E=\mathbb{R}^{n+1}$, the quadratic form $q=x_{1}^{2}-x_{2}^{2}-\cdots-x_ {n+1}^{2}$ and the canonical basis $e_{1}, e_ {2},\ldots, e_{n+1}$ of $\mathbb{R}^{n+1}$. The Clifford algebra $\mathscr{C}(q)$ can be then described as the $\mathbb{R}$-algebra generated by $e_{1}, e_ {2},\ldots, e_{n+1}$ with the relations $e_{1}^{2}=1$, $e_ {j}^{2}=-1$ for $j=2,\ldots, n+1$ and $e_{i}e_{j}=-e_{j}e_{i}$ for $i\neq j$.
The following lemma often will allow us to simplify the situation by conjugating in the group $\Spin(1,n)$.

\begin{lemma}\label{realpartinvariant}
The real part of the elements of $\Spin(1,n)$ is a conjugation invariant.
\end{lemma}

\begin{proof}

Consider the faithful representation 
$L:\Spin(1,n)\rightarrow \GL_{2^{n}}(\mathbb{R})$ given by the action of $\Spin(1,n)$ on $\mathscr{C}(q)$ by the left multiplication $L_{s}(x)=sx$. For the basis element $e_{M}$, the coefficient in $e_{M}$ of $L_{s}(e_{M})$ is equal to $s_{\mathbb{R}}$, hence the associated matrix of $L_{s}$ in this basis has all its entries in the principal diagonal equal to $s_{\mathbb{R}}$. This implies that the trace of $L_{s}$ is equal to $2^{n+1} s_{\mathbb{R}}$, and since the trace of matrices is conjugation invariant, it concludes the proof.\qedhere
\end{proof}

Now, considering the map $\varphi$ in \eqref{phi}, we can relate the displacement and the real part of elements in $\Spin(1,n)$. The main step is the following lemma.

\begin{lemma}\label{coshofdistance}
Let $s\in \Spin(1,n)$ and $\varphi_{s}$ its image under $\varphi$ in $\SO^{o}(1,n)$. If $A=(a_{i,j})_{i,j=1,\ldots,n+1}$ represents $\varphi_{s}$ in the basis $\lbrace e_{1},\ldots,e_{n+1}\rbrace$, we have $$\cosh(d(e_{1},\varphi_{s}(e_{1}))=a_{1,1}.$$

In particular, if $s=\sum_{M}s_{M}e_{M}$, then $\cosh(d(e_{1},\varphi_{s}(e_{1}))=\sum_{M} s_{M}^{2}$. 

\end{lemma}

\begin{proof}
The stabilizer
of the point $e_{1}=(1,0,\ldots, 0)\in\mathbb{H}^{n}$ for the action of $\SO(1,n)^{\circ}$ on $\mathbb{H}^{n}$ is given by the subgroup of matrices of the form

\begin{equation}\label{stabilizer1}
B=\begin{pmatrix}
  1 & 0 & 0 & \cdots & 0 \\
  0 & b_{2,1} & b_{2,2} & \cdots & b_{2,n+1}\\
  \vdots & \vdots  & \vdots  & \ddots & \vdots\\
  0 & b_{n,1} & b_{n,2} & \cdots & b_{n,n+1}\\
  \end{pmatrix},
\end{equation}

where the matrix $\tilde{B}=(b_{i,j})$ is an $n\times n$ orthogonal matrix.\\

Now, $\varphi_{s}(e_{1})=A(e_{1})=(a_{1,1},\ldots,a_{n+1,1})$ and we can find an orthogonal matrix $\tilde{B}$ such that  \\
\begin{center}
$\tilde{B}
 \begin{pmatrix}
  a_{2,1} \\
  a_{3,1} \\
  \vdots \\
  a_{n+1,1}\\
  
 \end{pmatrix}=
 \begin{pmatrix}
  0 \\
  0 \\
  \vdots \\
  c\\
  
 \end{pmatrix},$
where $c^{2}=a_{2,1}^{2}+\cdots+a_{n+1,1}^{2}.$ 
\end{center}

Then, choosing $B$ as in \eqref{stabilizer1} we have

\begin{align*}
\cosh(d(e_{1},\varphi_{s}(e_{1})) &=\cosh(d(Be_{1},B\varphi_{s}(e_{1}))\\ &
=\cosh(d(e_{1},a_{1,1}e_{1}+ce_{n+1}))\\ &
=a_{1,1}.
\end{align*}

The last equality holds due to the fact that the curve $\alpha(t)=\cosh(t)e_{1}+\sinh(t)e_{n+1}$ is a geodesic in $\mathbb{H}^{n}$ and then $d(e_{1},\alpha(t))=d(\alpha(0),\alpha(t))=t$ for any $t\in\mathbb{R}.$ So take $t$ such that $\cosh(t)=a_{1,1}$ and $\sinh(t)=c$. \\

Now, if $s=\sum_{M}s_{M}e_{M}$ then $se_{1}s^{*}=\sum_{M,N}s_{M}s_{N}e_{N} e_{1}e^{*}_{M}$. Note that if $N\neq M$ the product $e_{N}e_{1}e^{*}_{M}$ is a basis element different of $e_{1}$; if $N=M$ and $e_{M}$ contains $e_{1}$ in its product then $e_{N}e_{1}e^{*}_{M}=-e_{1}e_{M}e^{*}_{M}=e_{1}$; and if $N=M$ and $e_{M}$ does not contains $e_{1}$ in its product then $e_{N}e_{1}e^{*}_{M}=e_{1}e_{N}e^{*}_{N}=e_{1}$. With all this we deduce that the coordinate $e_{1}$ of $se_{1}s^{*}$, which is equal to $a_{1,1}$, agrees with $\sum_{M}s_{M}^{2}$.\qedhere

\end{proof}

We can conclude this section relating the displacement of an element $s\in\Spin(1,n)$ with its real part. This relation is a generalization of the well-known fact that in dimensions 2 or 3 the displacement of an isometry is related to the trace of the matrix in $\SL(2,\mathbb{R})$ or $\SL(2,\mathbb{C})$, see e.g. \cite[Chap. 7]{Beardon}.

\begin{proposition}\label{real part-displacement}
For any $s\in \Spin(1,n)$ with $\lvert s_{\mathbb{R}}\rvert\geq 1$ we have $$d(x,\varphi_{s}(x))\geq 2 \log(|s_{\mathbb{R}}|).$$
\end{proposition}

\begin{proof}
Since the real part is conjugation invariant, conjugating by an element in $\Spin(1,n)$ we can suppose that $x=e_{1}$. In this case, if $s=\sum_{M}s_{M}e_{M}$ by Lemma \ref{coshofdistance} we have

\begin{align*}
\cosh(d(e_{1},\varphi_{s}(e_{1}))=\sum_{M} s_{M}^{2}\geq s_{\mathbb{R}}^{2}.
\end{align*}
Hence $d(e_{1},\varphi_{s}(e_{1}))\geq \arccosh(s_{\mathbb{R}}^{2}) \geq 2\log(\lvert s_{\mathbb{R}}\rvert).$\qedhere
\end{proof}

\section{Lower bound for the displacement of $\Gamma(I)$}\label{lowerbound}
In this section we will study the displacement of the principal congruence subgroups $\Gamma(I)$ of $\Gamma=\Spin_{f}(\mathcal{O}_{k})$ in terms on the ideal $I$. The main ideas are inspired by \cite{KSV07} where the authors study the systole of compact hyperbolic surfaces and $3$-manifolds.\\

Suppose $f$ is an admissible quadratic form defined over a totally real number field of degree $d$ over $\mathbb{Q}$. We choose an orthogonal basis $\lbrace e_{1},\dots,e_{n+1}\rbrace$ such that all the coefficients of $f$ lie in the ring $\mathcal{O}_{k}$ and recall from Section \ref{congruencesubgroups}  that we can represent $\Gamma(I)$ in the form 

$$\Gamma(I)= \lbrace s=\sum_{\lvert M\rvert\hspace{1mm}\mbox{even}}s_{M} e_{M}\in\Gamma\mid s_{M}\in I \hspace{2mm}\mbox{for any}\hspace{2mm} M\neq \emptyset, s_{\mathbb{R}}-1\in I\rbrace.$$

By Proposition \ref{real part-displacement} one can find a lower bound for the displacement of elements in $\Gamma(I)$ from a lower bound for the real part of its elements. The condition $ss^{*}=1$ implies that

\begin{equation}\label{mainequality}
\sum_{ \lvert M\rvert \hspace{1mm}\mbox{even}}s_{M}^{2}f(e_{M})=1,
\end{equation}
where for $M=\lbrace i_{v_{1}},\ldots,i_{v_{k}}\rbrace$ we denote by $f(e_{M})$ the product\\ $f(e_{i_{v_{1}}})\cdots f(e_{i_{v_{k}}})$, which lies in $\mathcal{O}_{k}$.\\

If $s\in\Gamma(I)$, we obtain that $s_{\mathbb{R}}^{2}-1\in I^{2}$. Writing $s_{\mathbb{R}}=y_{0}+1$ with $y_{0}\in I$, we have $2y_{0}\in I^{2}$. From this we can obtain a lower bound for $\lvert s_{\mathbb{R}}\rvert$ in terms of the norm of the ideal $I$, $\N(I)$, which is by definition the cardinality of the quotient ring $\mathcal{O}_{k}/ I$.

\begin{lemma}\label{lowerboundrealpart}
For any non-trivial $s\in\Gamma(I)$, we have $\lvert s_{\mathbb{R}}\rvert \geq \frac{\N(I)^{2}}{2^{2d-1}}-1.$
\end{lemma}

\begin{proof}
Applying any non-trivial embedding $\sigma$ to Equation \eqref{mainequality}, since $f^{\sigma}$ is positive-definite, we have

\begin{equation}\label{applyingsigma}
\sigma(s_{\mathbb{R}})^{2}\leq \sum_{ M \hspace{1mm}\mbox{even}}\sigma(s_{M})^{2}f^{\sigma}(e_{M})=1.
\end{equation}

Replacing $\sigma(s_{\mathbb{R}})=\sigma(y_{0})+1$ in Equation \eqref{applyingsigma}, we obtain that $\lvert \sigma(y_{0})\rvert\leq 2$. Observe that $\sigma(y_{0})\neq 0$, since otherwise $\sigma(s_{\mathbb{R}})=1$ and Equation \eqref{applyingsigma} will imply $\sigma(s_{M})=0$ for any $M\neq\emptyset$, so by injectivity of $\sigma$ we will then have $s=1$.\\

Since $2y_{0}\in I^{2}$, we have $ \N(y_{0})\geq\frac{\N(I)^{2}}{2^{d}}$. By definition, $\N(y_{0})$ is equal to the product $\lvert y_{0}\rvert\prod_{\sigma\neq 1}\lvert\sigma(y_{0})\rvert$ and so $\lvert y_{0}\rvert\geq \frac{\N(I)^{2}}{2^{2d-1}}$. This shows that $$\lvert s_{\mathbb{R}}\rvert\geq \lvert y_{0}\rvert-1\geq \frac{\N(I)^{2}}{2^{2d-1}}-1.$$ 
\end{proof}

This result together with the results obtained in the previous section allow us to relate the displacement of elements of $\Gamma(I)$ with the norm $\N(I)$.

\begin{proposition}\label{mainproposition}
If $I\subset\mathcal{O}_{k}$ is an ideal with norm $\N(I)\geq 2^{d}$, then for any non-trivial $s\in\Gamma(I)$ and $x\in\mathbb{H}^{n}$ we have 

$$d(x,\varphi_{s}(x))\geq 4\log(\N(I))-4d\log(2).$$
\end{proposition}

\begin{proof}
The condition $\N(I)\geq 2^{d}$ implies that $\frac{\N(I)^{2}}{2^{2d-1}}-1\geq\frac{\N(I)^{2}}{2^{2d}}$ and so $\lvert s_{\mathbb{R}}\rvert\geq \frac{\N(I)^{2}}{2^{2d}}\geq 1$. The result follows from Proposition \ref{real part-displacement} and Lemma \ref{lowerboundrealpart}.
\end{proof}

This proposition shows in particular that the congruence subgroup $\Gamma(I)$ acts without fix points if the ideal $I$ has norm large enough, and so the quotient $M_{I}=\Gamma(I)\backslash\mathbb{H}^{n}$ is a hyperbolic manifold. 

\section{The index $[\Gamma:\Gamma(I)]$}\label{index}

In order to relate the systole of $\Gamma(I)\backslash\mathbb{H}^{n}$ with the index $[\Gamma:\Gamma(I)]$ we need to relate that index with the norm of the ideal $I$. To do that, we map $\Gamma$ to a suitable finite group depending on the Clifford algebra of $f$ and the ideal $I$. It will allow us to bound from above $[\Gamma:\Gamma(I)]$ by the cardinality of that finite group.\\

In this section we will assume that in the basis $\{e_{1},\ldots,e_{n+1}\}$ the admissible quadratic form $f$ is written in the form $f=a_{1}x_{1}^{2}-a_{2}x_{2}^{2}-\cdots-a_{n+1}x_{n+1}^{2}$ with $a_{i}\in\mathcal{O}_{k}$. Denote by $\mathcal{Q}$ the $\mathcal{O}_{k}$-algebra in $\mathscr{C}(f,\mathbb{R})$ given by $$\mathcal{Q}=\Bigg\{s=\sum_{\lvert M\rvert\hspace{1mm}\mbox{even}}s_{M} e_{M}\mid s_{M}\in \mathcal{O}_{k}\Bigg\}.$$

It is clear that the map $*$ preserves $\mathcal{Q}$, and any ideal $I\subset\mathcal{O}_{k}$ defines an ideal in $\mathcal{Q}$ given by

\begin{equation}\label{definition}
I\mathcal{Q}=\Bigg\{\sum_{\lvert M\rvert\hspace{1mm}\mbox{even}}s_{M} e_{M}\in \mathcal{Q}\mid s_{M}\in I\Bigg\} 
\end{equation}

which is also preserved by $*$. Therefore, we obtain a group automorphism in the quotient $\mathcal{Q}/I\mathcal{Q}$ given by 

\begin{align}
*: \mathcal{Q}/I\mathcal{Q}&\rightarrow \mathcal{Q}/I\mathcal{Q}, \\
s+I\mathcal{Q} & \mapsto s^{*}+I\mathcal{Q}. \nonumber
\end{align}

We consider now the finite group
\begin{equation}\label{definition2}
\left(\mathcal{Q}/I\mathcal{Q}\right)^{1}=\Bigg\{\bar{s}\in \mathcal{Q}/I\mathcal{Q}\mid \bar{s}\bar{s}^{*}=1 \hspace{1mm} \mbox{and}\hspace{1mm} \bar{s}\bar{E}\bar{s}^{*}=\bar{E} \Bigg\},
\end{equation}

where $\bar{s}$ denotes the image of $s$ via the projection map $\pi:\mathcal{Q}\rightarrow \mathcal{Q}/I\mathcal{Q}$ and $\bar{E}=\Bigg\{\sum_{i=1}^{n+1}\bar{x}_{i}\bar{e}_{i}\in \mathcal{Q}/I\mathcal{Q}\mid x_{i}\in \mathcal{O}_{k}\Bigg\}.$ 

\begin{proposition}
For any ideal $I\subset\mathcal{O}_{k}$ we have that $$[\Gamma:\Gamma(I)]\leq \lvert \left(\mathcal{Q}/I\mathcal{Q}\right)^{1}\rvert.$$
\end{proposition}

\begin{proof}
It was observed in the end of Section \ref{congruencesubgroups} that the group $\Gamma=\Spin_{f}(\mathcal{O}_{k})$ can be represented in the form

$$\Gamma=\Bigg\{ s\in\mathcal{Q}\mid sEs^{*}=E  \hspace{2mm}\mbox{and}\hspace{2mm} ss^{*}=1 \Bigg\}.$$

By definition, the projection map $\pi:\mathcal{Q}\rightarrow \mathcal{Q}/I\mathcal{Q}$ reduces to a group homomorphism $\pi|_{\Gamma}:\Gamma\rightarrow \left(\mathcal{Q}/I\mathcal{Q}\right)^{1}$ with kernel equal to $\Gamma(I).$ Then, for any ideal $I\subset\mathcal{O}_{k}$ we obtain that $$[\Gamma:\Gamma(I)]\leq \lvert \left(\mathcal{Q}/I\mathcal{Q}\right)^{1}\rvert.$$

\end{proof}

Before bounding from above the cardinality of $\left(\mathcal{Q}/I\mathcal{Q}\right)^{1}$, we require an auxiliary result.

\begin{lemma}
Let $\mathbb{F}$ be a finite field, $n\geq 2$ and $f=b_{1}x_{1}^{2}+\cdots+b_{n+1}x_{n+1}^{2}$ be a quadratic form with $b_{i}\in\mathbb{F}^{\times}$, then $f$ is non-degenerate and isotropic.
\end{lemma}

\begin{proof}
The fact that $f$ is non-degenerate follows from the diagonal form of $f$. For the second part it is enough to prove that the quadratic form $g=b_{1}x_{1}^{2}+b_{2}x_{2}^{2}+b_{3}x_{3}^{2}$ is universal. For $c\in\mathbb{F}$ the sets $A=\{b_{1}+b_{2}y^{2}\mid y\in\mathbb{F}\}$ and $B=\{c-b_{3}z^{2}\mid z\in\mathbb{F}\}$ have the same cardinality, which it is equals to $\frac{\lvert\mathbb{F}\rvert+1}{2}$, therefore $A\cap B\neq\emptyset$ and there exist $y_{2}$ and $y_{3}$ in $\mathbb{F}$ such that

$$b_{1}y_{1}^{2}+b_{2}y_{2}^{2}+b_{3}y_{3}^{2}=c$$

with $y_{1}=1$.\qedhere

\end{proof}

\begin{proposition}\label{finiteset}
There exists a finite set $S$ of prime ideals in $\mathcal{O}_{k}$ such that for any ideal $I$ which is not divisible for any element in $S$ we have
$$\lvert \left(\mathcal{Q}/I\mathcal{Q}\right)^{1}\rvert\leq\N(I)^{\frac{n(n+1)}{2}}.$$

In particular, for those ideals we have that $[\Gamma:\Gamma(I)]\leq\N(I)^{\frac{n(n+1)}{2}}$.
\end{proposition}

\begin{proof}
Since $f=a_{1}x_{1}^{2}-a_{2}x_{2}^{2}-\cdots-a_{n+1}x_{n+1}^{2}$, with $a_{i}\in\mathcal{O}_{k}$ the discriminant $D_{f}$ of $f$ is equal to the product of its coefficients. Put $D=2D_{f}$ and let $S$ be the set of prime ideals which contain $D$.\\

Suppose that $I=\prod\mathfrak{p}_{i}^{r_{i}}$ is the decomposition of the ideal $I$ in prime ideals, and that none of the factors $\mathfrak{p}_{i}$ belongs to $S$. By the Chinese Remainder Theorem $\mathcal{Q}/I\mathcal{Q}\cong\prod\mathcal{Q}/\mathfrak{p}_{i}^{r_{i}}\mathcal{Q}$ and since the projection on each of the components preserves the map $*$ this isomorphism restricts to    $(\mathcal{Q}/I\mathcal{Q})^{1}\cong\prod(\mathcal{Q}/\mathfrak{p}_{i}^{r_{i}}\mathcal{Q})^{1}$. We have reduced to prove the result for ideals of the form $I=\mathfrak{p}^{r}$ with $D\notin\mathfrak{p}$, which is equivalent to none of the coefficients $a_{1},\ldots,a_{n+1}$ neither $2$ belongs to $\mathfrak{p}$. For those ideals we will prove the result by induction on $r$.\\

For $r=1$, follows from the equations defining $(\mathcal{Q}/\mathfrak{p}\mathcal{Q})^{1}$ that this group coincides with the spin group $\Spin_{\bar{f}}(\mathbb{F}_{\N(\mathfrak{p})})$ of the quadratic form $\bar{f}$ over the finite field with $\N(\mathfrak{p})$ elements $\mathbb{F}_{\N(\mathfrak{p})}$, where $\bar{f}$ denotes the reduction modulo $\mathfrak{p}$ of $f$. Since none of the coefficients $a_{1},\ldots,a_{n+1}$ belongs to $\mathfrak{p}$ we have that $\overline{f}$ is a non-degenerate and isotropic form by the previous lemma. Since $2\notin\mathfrak{p}$ then char($\mathbb{F}_{\N(\mathfrak{p})}$)$\neq 2$ and we saw in the end of Section \ref{Clifford algebras and spin group} that $\lvert\SO_{\bar{f}}(\mathbb{F}_{\N(\mathfrak{p})})\rvert=\lvert\Spin_{\bar{f}}(\mathbb{F}_{\N(\mathfrak{p})})\rvert$. Now, it is known that the cardinality of the special orthogonal group $\SO_{\bar{f}}(\mathbb{F}_{\N(\mathfrak{p})})$ is bounded from above by $\N(\mathfrak{p})^{\frac{n(n+1)}{2}}$ (see e.g.  \cite[Sec. 3.7.2]{Wilson09}), then

\begin{equation*}
\lvert (\mathcal{Q}/\mathfrak{p}\mathcal{Q})^{1}\rvert=\lvert \Spin_{\bar{f}}\left(\mathbb{F}_{\N(\mathfrak{p})}\right)\rvert
=\lvert\SO_{\bar{f}}\left(\mathbb{F}_{\N(\mathfrak{p})}\right)\rvert\leq \N(\mathfrak{p})^{\frac{n(n+1)}{2}}.
\end{equation*}

Assume now that the result holds for $I=\mathfrak{p}^{r}$ with $r>1$. Consider the natural map

\begin{align*}
\theta: \mathcal{Q}/\mathfrak{p}^{r+1}\mathcal{Q}&\rightarrow \mathcal{Q}/\mathfrak{p}^{r}\mathcal{Q} \\
\bar{s} & \mapsto \bar{s}.
\end{align*}

which is well-defined because $\mathfrak{p}^{r+1}\subset\mathfrak{p}^{r}$. By \eqref{definition2} this map restricts to a group homomorphism $\theta|_{(\mathcal{Q}/\mathfrak{p}^{r+1}\mathcal{Q})^{1}}:(\mathcal{Q}/\mathfrak{p}^{r+1}\mathcal{Q})^{1}\rightarrow(\mathcal{Q}/\mathfrak{p}^{r}\mathcal{Q})^{1}$. To prove the result for $r+1$ we will determine the kernel of $\theta|_{(\mathcal{Q}/\mathfrak{p}^{r+1}\mathcal{Q})^{1}}$ and we will prove that $\lvert\ker(\theta|_{(\mathcal{Q}/\mathfrak{p}^{r+1}\mathcal{Q})^{1}})\rvert=\N(\mathfrak{p})^\frac{n(n+1)}{2}$.\\

If $\bar{s}\in(\mathcal{Q}/\mathfrak{p}^{r+1}\mathcal{Q})^{1}$ with $\theta(\bar{s})=\bar{1}$, then $\bar{s}=\bar{1}+\bar{t}$ with $\bar{t}\in\mathfrak{p}^{r}\mathcal{Q}/\mathfrak{p}^{r+1}\mathcal{Q}$. Replacing $\bar{s}=\bar{1}+\bar{t}$ in the equation $\bar{s}\bar{s}^{*}=1$ we have that $\bar{t}+\bar{t}^{*}=0$ in $\mathcal{Q}/\mathfrak{p}^{r+1}\mathcal{Q}$. In particular $t_{\emptyset}\in\mathfrak{p}^{r+1}$ because $2\notin\mathfrak{p}^{r+1}$. Now, the condition $\bar{s}\bar{E}\bar{s}^{*}=\bar{E}$ implies that 

$$\bar{s}\bar{e_{i}}\bar{s}^{*}=(\bar{1}+\bar{t})\bar{e_{i}}(\bar{1}-\bar{t})=\bar{e_{1}}-\bar{e_{i}}\bar{t}+
\bar{t}\bar{e_{i}}\in\bar{E}$$  
for any $i=1,\ldots,n+1.$ Then the expression $\bar{e_{i}}\bar{t}-\bar{t}\bar{e_{i}}$ lies in $\bar{E}$ for any $i=1,\dots,n+1$. Now, if we write $\bar{t}=\sum_{\lvert M\rvert\hspace{1mm}\mbox{even}}\bar{t}_{M} \bar{e}_{M}$ we have by direct computation that

\[
\bar{t}_{M}(\bar{e_{i}}\bar{e}_{M}-\bar{e}_{M}\bar{e_{i}})=
\begin{cases}

\text{0,} &\quad\text{if $i\notin M.$}\\
\text{$2\bar{t}_{M}\bar{e_{i}}\bar{e}_{M}$,} &\quad\text{if $i\in M.$}\
\end{cases}
\]

Since $\bar{e_{i}}\bar{e}_{M}\notin\bar{E}$ for $\lvert M\rvert\geq 4$ and $2\notin\mathfrak{p}^{r+1}$ we have that $\bar{t}_{M}=0$ in $\mathcal{Q}/\mathfrak{p}^{r+1}\mathcal{Q}$ if $\lvert M\rvert\geq 4.$ From this we obtain that

$$\ker(\theta|_{(\mathcal{Q}/\mathfrak{p}^{r+1}\mathcal{Q})^{1}})=\Bigg\{ \bar{1}+\sum_{\lvert M\rvert=2}\bar{t}_{M}\bar{e}_{M}\mid \bar{t}_{M}\in \mathfrak{p}^{r}/\mathfrak{p}^{r+1}\Bigg\}.$$

Since the cardinality of the quotient $\mathfrak{p}^{r}/\mathfrak{p}^{r+1}$ is equal to $\N(\mathfrak{p})$ we have that $\lvert\ker(\theta|_{(\mathcal{Q}/\mathfrak{p}^{r+1}\mathcal{Q})^{1}})\rvert=\N(\mathfrak{p})^{\frac{n(n+1)}{2}}$. By induction we conclude that

$$\lvert(\mathcal{Q}/\mathfrak{p}^{r+1}\mathcal{Q})^{1}\rvert\leq\lvert\ker(\theta|_{(\mathcal{Q}/\mathfrak{p}^{r+1}\mathcal{Q})^{1}})\rvert\lvert(\mathcal{Q}/\mathfrak{p}^{r}\mathcal{Q})^{1}\rvert\leq\N(\mathfrak{p}^{r+1})^{\frac{n(n+1)}{2}}.$$
\end{proof}

\section{Proof of the main theorem}\label{proofoftheorem}

We can now to put all the pieces together and prove the theorem.

\begin{theorem}\label{maintheorem}
Let $\Gamma$ be an arithmetic subgroup of $\Spin(1,n)$ of the first type defined over a totally real number field $k$. There exists a finite set S of prime ideals in $\mathcal{O}_{k}$ such that for any sequence of ideals $I$ with no prime factors in $S$ the principal congruence subgroups $\Gamma(I)$ eventually satisfy

$$\sys(M_{I})\geq \frac{8}{n(n+1)}\log(\vol(M_{I}))-c,$$
where $M_{I}=\Gamma(I)\backslash \mathbb{H}^{n}$ and $c$ is a constant independent on $I$.
\end{theorem}

\begin{proof}
Without loss of generality we can assume that there exists an admissible quadratic form $f$ which is diagonal with coefficients in the ring $\mathcal{O}_{k}$ and $\Gamma=\Spin_{f}(\mathcal{O}_{k})$. These assumptions may require passing to a finite sheet covering, which is compatible with the statement of the theorem.\\

Now, if $\alpha$ is a closed geodesic on $M_{I}=\Gamma(I)\backslash\mathbb{H}^{n}$ with the length equal to $\sys(M_{I})$, then there exists an element $s\in\Gamma(I)$ and a point $x\in\mathbb{H}^{n}$ such that $d(x,\varphi_{s}(x))=\ell(\alpha)=\sys(M_{I})$. If the norm if $I$ is large enough, Proposition \ref{mainproposition} implies that

$$\sys(M_{I})\geq 4\log(\N(I))-4d\log(2).$$

Choosing the set $S$ and the ideals $I$ as in Proposition \ref{finiteset} we conclude that 

$$\sys(M_{I})\geq \frac{8}{n(n+1)}\log([\Gamma:\Gamma(I)])-4d\log(2).$$

Since $\vol(M_{I})=[\Gamma:\Gamma(I)]\vol (M)$ with $M=\Gamma\backslash\mathbb{H}^{n}$, then if the norm of the ideal $I$ is large enough, we obtain the desired inequality with $c=4d\log(2)+\frac{8}{n(n+1)}\log(\vol(M))$. The systole for other $I$ can be compensated by enlarging the constant
\end{proof}

\textit{Remark}: As we mentioned in the introduction, the systole of congruence coverings of a compact arithmetic hyperbolic manifold $M$ has logarithmic growth with respect to the volume of $M$. For any compact hyperbolic manifold, an upper bound for the systole can be obtained from the fact that the injectivity radius is equal to $\frac{\sys(M)}{2}$, but this argument does not apply in the non-compact case. In particular, if $k=\mathbb{Q}$ and $n\geq 4$ all the examples considered in this work are non-cocompact and we would like to have an upper bound for $\sys(M_{I})$ in terms of logarithm of the volume. Recently, M. Gendulphe \cite{Gen15} overcame this problem showing that for any non-compact hyperbolic manifold with finite volume $$\sys(M)\leq 2\log(\vol(M))+d,$$ with $d$ an explicit constant depending on the dimension of $M$. For principal congruence coverings of a compact arithmetic hyperbolic manifold of the first type we will obtain a better upper bound of the systole in the next section (Corollary \ref{upperbound}).

\section{Systolic genus of hyperbolic manifolds}\label{systolicgenus}

Historically, sequences of congruence coverings have been used in many different contexts. We would like to apply Theorem \ref{maintheorem} to improve some results obtained in the last years. The first application concerns to the systolic genus of arithmetic hyperbolic manifolds studied by M. Belolipetsky in \cite{Bel13}.\\

If we denote by $S_{g}$ a Riemann surface of genus $g\geq 1$, then the \textit{systolic genus} of a Riemannian manifold $M$ is defined by 
$$\sysg(M)=\min\lbrace g \mid \pi_{1}(M)\hspace{1mm} \mbox{contains} \hspace{1mm} \pi_{1}(S_{g})\rbrace.$$

The main result in \cite{Bel13} relates the systolic genus $\sysg(M)$ of a hyperbolic manifold $M$ with the systole $\sys(M)$.

\begin{theorem}{\cite[Theorem 5.1]{Bel13}}.\label{beltheorem1}
Let $M$ be a closed $n$-dimensional hyperbolic manifold. For any $\epsilon >0$, assuming that $\sys(M)$ is sufficiently large, we have

\begin{equation*}
\sysg(M)\geq e^{(\frac{1}{2}-\epsilon)\sys(M)}.\\
\end{equation*}

\end{theorem}

Concerning the congruence coverings in dimension $n\geq 3$, the following result is proved in \cite{Bel13}:

\begin{proposition}{\cite[Proposition 5.3]{Bel13}}.\label{beltheorem2}
Let $\Gamma$ be a fundamental group of a closed arithmetic hyperbolic manifold of dimension $n\geq 3$.

\begin{enumerate}[label=(\Alph*)]
\item There exists constant $C>0$ such that for a decreasing sequence $\Gamma_{i}<\Gamma$ of congruence subgroups of $\Gamma$, the corresponding quotient manifolds $M_{i}=\Gamma_{i}\backslash\mathbb{H}^{n}$ satisfy
\begin{equation*}
\log\sysg(M_{i})\gtrsim C\log(\vol(M_{i})), \mbox{as}\hspace{2mm} i\rightarrow\infty.
\end{equation*}

\item If $\Gamma$ is of the first type, the sequence of principal congruence subgroups $\Gamma_{I}$ associated to prime ideals satisfy
\begin{equation*}
\sysg(M_{I})\lesssim\vol(M_{I})^{\frac{6}{n(n+1)}.},\hspace{2mm} \mbox{as}\hspace{2mm} \N(I)\rightarrow\infty,
\end{equation*}

where $M_{I}=\Gamma_{I}\backslash\mathbb{H}^{n}.$

\end{enumerate}
\end{proposition}

We recall that for two positive functions $f(x)$ and $g(x)$, the relation $f(x)\gtrsim g(x)$ means that for any $\epsilon >0$ there exists $x_{0}$ depending on $\epsilon$ such that $f(x)\geq (1-\epsilon)g(x)$ for all $x\geq x_{0}$.\\

The explicit constant $C=\frac{1}{3}$ is known in dimension $n=3$ \cite[Proposition 3.1]{Bel13}, but in higher dimensions no explicit value of this constant was known so far. In this respect, we can apply Theorem \ref{maintheorem} to give a quantitative version of this result.

\begin{proposition}\label{firstapplication}
Let $\Gamma$ be a fundamental group of a closed arithmetic hyperbolic manifold of the first type of dimension $n\geq 3$. Then $\Gamma$ has a sequence of congruence subgroups $\Gamma_{i}$ such that the quotient manifolds $M_{i}=\Gamma_{i}\backslash\mathbb{H}^{n}$ satisfy

\begin{equation*}
\log\sysg(M_{i})\gtrsim \frac{4}{n(n+1)}\log(\vol(M_{i})), \hspace{4mm} \mbox{as}\hspace{2mm} i\rightarrow\infty.
\end{equation*}
\end{proposition}

\begin{proof}

 By Theorem \ref{maintheorem}, there exist a sequence of principal congruence subgroups $\Gamma_{I}<\Gamma$ such that  $$\sys(M_{I})\gtrsim\frac{8}{n(n+1)}\log(\vol(M_{I})), \hspace{2mm} \mbox{as}\hspace{2mm} \N(I)\rightarrow\infty$$
The result now follows from Theorem \ref{beltheorem1}.\qedhere

\end{proof}

Joining together Theorem \ref{beltheorem1} with Part (B) of Proposition \ref{beltheorem2} we can obtain an upper bound for the systole of the manifolds considered in Theorem \ref{maintheorem} in the prime case. 

\begin{corollary}\label{upperbound}
Let $n\geq 3$ and let $\Gamma$ be an arithmetic subgroup of $\Spin(1,n)$ of the first type defined over a totally real number field $k$. Then for any sequence of prime ideals $I\subset\mathcal{O}_{k}$ the principal congruence subgroups $\Gamma(I)$ eventually satisfy

\begin{equation*}
\frac{8}{n(n+1)}\log\left(\vol(M_{I})\right)\lesssim
\sys\left(M_{I}\right)\lesssim\frac{12}{n(n+1)}\log\left(\vol(M_{I})\right),
\end{equation*}  
 as $\N(I)\rightarrow\infty$, where $M_{I}=\Gamma(I)\backslash \mathbb{H}^{n}.$
\end{corollary}

\section{Homological Codes}\label{homologicalcodes}

In this section we apply Theorem \ref{maintheorem} to homological codes. Motivated by a question of Z\'emor \cite{Zemir09}, L. Guth and A. Lubotzky constructed in 2013 a certain class of homological codes using congruence coverings of arithmetic hyperbolic 4-dimensional manifolds \cite{GL13}. For the definition of homological codes and the details of the construction using hyperbolic manifolds we refer the reader to \cite{GL13} and the references therein.\\

According to the known examples, Z\'emor \cite{Zemir09} asked if it is true that every $[[n,k,d]]$ homological quantum code satisfies the inequality $kd^{2}\leq n^{1+o(1)}$. E. Fetaya \cite{Fetaya12} proved that it holds for surfaces but L. Guth and A. Lubotzky \cite{GL13} gave counterexamples in dimension 4. The construction comes from congruence coverings of a compact $4$-dimensional arithmetic hyperbolic manifold. The main result in \cite{GL13} can be stated in the following way.

\begin{theorem}{\cite[Theorem 1]{GL13}}\label{GLtheorem1}
Let $M$ be a compact arithmetic hyperbolic $4$-manifold. There exist constants $\epsilon, \epsilon_{1}, \epsilon_{2}>0$ such that for a sequence of congruence coverings $M_{I}$ with triangulations $X$ the associated homological quantum codes constructed in $C^{2}\left(X,\mathbb{Z}_{2}\right)=C_{2}\left(X,\mathbb{Z}_{2}\right)$ are $[[n,\epsilon_{1}n,n^{\epsilon_{2}}]]$ codes and satisfy

$$kd^{2}\geq n^{1+\epsilon}.$$ 

\end{theorem}

Related to Z\'emor's question, it could be interesting to obtain an explicit value of the constant $\epsilon>0$ in the previous result. Since the construction in \cite{GL13} make use of congruence coverings of arithmetic 4-dimensional hyperbolic manifolds, we can use Theorem \ref{maintheorem} to give an explicit value of this constant.\\

As we noted in Section \ref{Arithmetic subgroups of Spin(1,n)}, the fundamental group $\Gamma=\pi_{1}(M)$ embeds as an arithmetic subgroup of $\Spin(1,4)$ of the first type and it is defined over a totally real number field $k\neq\mathbb{Q}$. By compactness, the injectivity radius of $M_{I}$ is equal to $\frac{\sys(M_{I})}{2}$,
therefore taking ideals $I\subset \mathcal{O}_{k}$ as in Theorem \ref{maintheorem} with norm sufficiently large, the injectivity radius of $M_{I}$ satisfy
$$\inj(M_{I})\geq\frac{1}{5}\log(\vol(M_{I}))-c,$$

for some constant $c$ independent of $M_{I}$. Using this bound in the proof of Theorem \ref{GLtheorem1} in \cite{GL13}, we obtain that the distance of the codes in this construction satisfy $$d\geq c_{1}n^{0.2},$$

for some positive constant $c_{1}$. It is known that $d=O(n^{0.3})$ \cite[Remark 20]{GL13}, hence this bound is quite close to the optimal one. To conclude with the estimate in Theorem \ref{GLtheorem1}, recall that the dimension $k$ of the code satisfy $$k\geq c_{2}n,$$

for some positive constant $c_{2}$ if the norm of the ideal $I$ is sufficiently large \cite[Theorem 6]{GL13}. Therefore, we obtain that for those ideals $I$ with norm sufficiently large, the homological codes constructed from $M_{I}$ satisfy $$kd^{2}\geq c_{3}n^{1+0.4},$$

for some positive constant $c_{3}.$

\appendix
\section{}
\begin{center}
Cayo Doria and Plinio G. P. Murillo\\
\end{center}

Let $\Gamma$ be an arithmetic subgroup of the first type in $\Isom^{+}(\mathbb{H}^{n})$ defined over a totally real number field $k$, and denote by $\Gamma(I)$ the principal congruence subgroup associated to the ideal $I\subset\mathcal{O}_{k}$ (see the Section \ref{Arithmetic subgroups of Spin(1,n)} of this paper for the definitions). For two positive functions $f(x)$ and $g(x)$ we write $f(x)\gtrsim g(x)$ if for any $\epsilon >0$ there exists $x_{0}$ depending on $\epsilon$ such that $f(x)\geq (1-\epsilon)g(x)$ for all $x\geq x_{0}$. We write $f(x)\sim g(x)$ if $f(x)\gtrsim g(x)$ and $g(x)\gtrsim f(x)$. The purpose of this appendix is to prove the following result.

\begin{theorem}\label{appendix}
Up to passing to a commensurable group, for any sequence of prime ideals $I\subset\mathcal{O}_{k}$ the principal congruence subgroups $\Gamma(I)$ satisfy

$$\sys(M_{I})\lesssim \frac{8}{n(n+1)}\log(\vol(M_{I})),$$

where $M_{I}=\Gamma(I)\backslash\mathbb{H}^{n}$.
\end{theorem}

Comparing this result with Theorem \ref{maintheorem} we conclude that the multiplicative constant $\frac{8}{n(n+1)}$ is sharp. This result was proved by S. Makisumi in 2012 in dimension $n=2$ \cite{Mak13}. In that case, the arithmetic group $\Gamma$ can be constructed as integral points with norm one in a quaternion algebra over a totally real number field (see \cite[Section 10.2]{MR03} for more details about this approach). The problem of finding short closed geodesics translates to finding elements with norm equal to one in a quaternion algebra satisfying suitable modular conditions, and with small trace. To find some of these elements Makisumi appeals to number theoretic local-global techniques. 
We will combine this result with the geometry of $M_{I}$ to give a proof in any dimension.

\begin{proof}[Proof of Theorem \ref{appendix}]
By passing to a commensurable group we can suppose that $\Gamma=\Spin_{f}(\mathcal{O}_{k})$, where $f=a_1 x_1^2 - a_2x_2^2 - \cdots - a_{n+1}x_{n+1}^2$  is a diagonal quadratic form in the basis $\{e_{1}, e_{2},\ldots,e_{n+1}\}$ with $a_{i} > 0$ in $\mathcal{O}_{k}$, such that for any non-trivial Galois embeddings $\sigma: k \to \mathbb{R}$ we have $\sigma(a_1)>0 $ and $\sigma(a_i)<0 $ for $i=2,\cdots,n+1$. Let $E' \subset E$ be the 3-dimensional subspace generated by \{$e_1,e_2,e_3$\}. Note that the restriction $f':E' \to k $ of $f$ to $E'$ has signature $(1,2)$. In this case we have a natural inclusion $\Spin_{f'}(\mathcal{O}_{k}) \hookrightarrow \Gamma$ and $\Gamma'=\Spin_{f'}(\mathcal{O}_{k})$ is an arithmetic lattice of $\Spin_{f'}(\mathbb{R})$. Consider an isometric embedding of $\mathbb{H}^2$ to $\mathbb{H}^n$ equivariant for the respective actions of $\Gamma'$ and $\Gamma$ and the inclusion above. For any  ideal $I \subset \mathcal{O}_{k}$ we obtain a totally geodesic embedding  $$S_{I}  \hookrightarrow M_{I} $$ where $S_{I}= \Gamma'(I) \backslash \mathbb{H}^2$. In particular, this is a $\pi_1$-injective embedding and then
$$\sys(M_{I}) \leq \sys(S_{I}).$$
On the other hand, the even Clifford algebra $\mathscr{C}^{+}(f',k)$ is the quaternion algebra $A=\left( \frac{a,b}{k} \right)$ with $a=a_1a_2$ and $b=a_1a_3 \in \mathcal{O}_{k}$ (\textit{cf}. Section \ref{Clifford algebras and spin group}). Moreover the group $\Gamma'$ coincide with the group of units of the order $A(\mathcal{O}_{k})$ and  $\Gamma'(I)$ is the kernel of the projection map $ A(\mathcal{O}_{k}) \to A(\mathcal{O}_{k}/I)$. With this identification $\Gamma'$ acts on both of the models of the hyperbolic plane in such a way that acts on the  upper half space model via the quaternion algebra $A$ and acts on the hyperboloid model via $\Spin_{f'}(\mathbb{R})$. It is known that there exist an isometry between this two models which is equivariant for the actions of $\Gamma'$ \cite[Proposition 5.3]{EGM87}, 
and applying \cite[Theorem 1.6]{Mak13} to the sequence $S_{I}$ we have 

$$\sys(S_{I}) \lesssim \dfrac{4}{3} \log(\area(S_{I})).$$

We now follow some arguments as in \cite[Proposition 3.2]{Bel13}. 
For $\N(I)$ large enough $\vol(M_{I})=\nu|\Spin_{\bar{f}}(\mathcal{O}_{k}/I)|$, where 
$\nu=\vol(\Gamma\backslash\mathbb{H}^{n})$. Restricting to prime ideals $I$ the projection of the quadratic
form $f$ to the finite field $\mathcal{O}_{k}/I$ is non-degenerate and $|\Spin_{\bar{f}}(\mathcal{O}_{k}/I)|=|\SO_{\bar{f}}(\mathcal{O}_{k}/I)|$. By \cite[Sec. 3.7.2]{Wilson09} we have that  

$$ \vol(M_{I}) \sim \nu\N(I)^{\frac{n(n+1)}{2}}.$$

In the same way
$$  \area(S_{I}) \sim \mu\N(I)^3$$
 
 where $\mu=\area(\Gamma'\backslash\mathbb{H}^2).$ Follows from this that $$\log(\area(S_{I})) \sim \frac{6}{n(n+1)} \log(\vol(M_{I}))$$ and therefore $\sys(M_{I}) \lesssim \frac{8}{n(n+1)} \log(\vol(M_{I}))$.\qedhere

\end{proof}


\bibliographystyle{plain}
\bibliography{mybiblio}
\vspace{4mm}
Plinio G. P. Murillo\\
Universit\"at Bern, Mathematisches Institut, Sidlerstrasse 5, CH-3012 Bern, Switzerland.\\
\textit{E-mail address}:\hspace{2mm}\texttt{plinio.pino@math.unibe.ch}

\vspace{6mm}
Cayo Dória\\
IMPA. Estrada Dona Castorina 110, 22460-320, Rio de Janeiro, Brazil.\\
\textit{E-mail address}:\hspace{2mm}\texttt{crfdoria@impa.br}

\end{document}